\documentclass[11pt]{amsart}

\usepackage{amsmath, amsthm, amssymb}

\newtheorem{thm}{Theorem}
\newtheorem{prop}[thm]{Proposition}
\newtheorem{lem}[thm]{Lemma}

\def\Re{\text{\rm Re\,}}

\begin{document}

\title[Squares of positive $(p,p)$-forms]
{Squares of positive $(p,p)$-forms}

\def\nl{\newline\phantom{aa}\hskip 8pt}

\author[Z. B\l ocki]{Zbigniew B\l ocki}
\address{Instytut Matematyki, Uniwersytet Jagiello\'nski, \L ojasiewicza 6,
30-348 Krak\'ow, Poland}
\email{Zbigniew.Blocki@im.uj.edu.pl, umblocki@cyf-kr.edu.pl}
\author[Sz. Pli\'s]{Szymon Pli\'s}
\address{Instytut Matematyki, Politechnika Krakowska, Warszawska 24, 
31-155 Krak\'ow, Poland} 
\email{splis@pk.edu.pl}
\thanks{The second named author was partially supported by the NCN grant
2011/01/D/ST1/04192.}

\begin{abstract}
We show that if $\alpha$ is a positive $(2,2)$-form then so is 
$\alpha^2$. We also prove that this is no longer true for forms
of higher degree.
\end{abstract}

\makeatletter
\@namedef{subjclassname@2010}{%
  \textup{2010} Mathematics Subject Classification}
\makeatother

%\subjclass[2010]{}

%\keywords{}

\maketitle

\section*{Introduction}

Recall that a $(p,p)$-form $\alpha$ in $\mathbb C^n$ is called 
{\it positive} (we write $\alpha\geq 0$) if for $(1,0)$-forms 
$\gamma_1,\dots,\gamma_{n-p}$ one has 
  $$\alpha\wedge i\gamma_1\wedge\bar\gamma_1\wedge\dots\wedge
       i\gamma_{n-p}\wedge\bar\gamma_{n-p}\geq 0.$$
This is a natural geometric condition, positive $(p,p)$-forms
are for example characterized by the following property: for every 
$p$-dimensional subspace $V$ and a test function $\varphi\geq 0$ one has
  $$\int_V\varphi\,\alpha\geq 0.$$

It is well known that positive forms are real (that is $\bar\alpha=\alpha$) 
and if $\beta$ is a (1,1)-form then
\begin{equation}\label{cond}
\alpha\geq 0,\ \beta\geq 0\ \Rightarrow\ \alpha\wedge\beta\geq 0.
\end{equation}
It was shown by Harvey and Knapp \cite{HK} (and independently by Bedford 
and Taylor \cite{BT}) that \eqref{cond} does not hold for all $(p,p)$
and $(q,q)$-forms $\alpha$ and $\beta$, respectively. We refer to
Demailly's book \cite{De}, pp. 129-132, for a nice and simple 
introduction to positive forms.

Dinew \cite{D} gave an explicit example of $(2,2)$-forms 
$\alpha,\beta$ in $\mathbb C^4$ such that $\alpha\geq 0$, $\beta\geq 0$ 
but $\alpha\wedge\beta<0$. We will recall it in the next section.
The aim of this note is to show the following, somewhat surprising result:

\begin{thm}\label{thm1}
Assume that $\alpha$ is a positive $(2,2)$-form. Then $\alpha^2$ is also 
positive.
\end{thm}

It turns out that this phenomenon holds only for $(2,2)$-forms:

\begin{thm}\label{thmp}
For every $p\geq 3$ there exists a $(p,p)$-form $\alpha$ in $\mathbb C^{2p}$
such that $\alpha\geq 0$ but $\alpha^2<0$.
\end{thm}

We do not know if similar results hold for higher powers of positive forms.

The paper is organized as follows: in Section \ref{s1} we present Dinew's 
criterion for positivity of (2,2)-forms in $\mathbb C^4$ which reduces 
the problem to a certain property of $6\times 6$ matrices. Further 
simplification reduces the problem to $4\times 4$ matrices. We then solve 
it in Section \ref{s2}. This is the most technical part of the paper.
Higher degree forms are analyzed in Section \ref{s3} where a counterpart 
of Dinew's criterion is showed and Theorem \ref{thmp} is proved.

\section{Dinew's Criterion}\label{s1}

Without loss of generality we may assume that $n=4$. Let $\omega_1,\dots,
\omega_4$ be a basis of $(\mathbb C^4)^\ast$ such that
  $$dV:=i\omega_1\wedge\bar\omega_1\wedge\dots\wedge i\omega_4\wedge
     \bar\omega_4=\omega_1\wedge\dots\wedge\omega_4\wedge
     \bar\omega_1\wedge\dots\wedge\bar\omega_4>0.$$
Set
  $$\Omega_1:=\omega_1\wedge\omega_2,\ \ \ \Omega_2:=\omega_1\wedge\omega_3,\ \ \ 
    \Omega_3:=\omega_1\wedge\omega_4,$$
  $$\Omega_4:=\omega_2\wedge\omega_3,\ \ \ \Omega_5:=-\omega_2\wedge\omega_4,\ \ \
   \Omega_6:=\omega_3\wedge\omega_4.$$
   Then 
  $$\Omega_j\wedge\Omega_k=
    \begin{cases}\omega_1\wedge\dots\wedge\omega_4,&\text{if }k=7-j\\
     0,&\text{otherwise.}\end{cases}$$
Every $(2,2)$-form $\alpha$ we can associate with a $6\times 6$-matrix 
$A=(a_{jk})$ by
  $$\alpha=\sum_{j,k}a_{jk}\,\Omega_j\wedge\bar\Omega_k.$$
For
  $$\beta=\sum_{j,k}b_{jk}\,\Omega_j\wedge\bar\Omega_k$$
  we have
\begin{equation}\label{prod}
\alpha\wedge\beta=\sum_{j,k}a_{jk}\,b_{7-j,7-k}\,dV.
\end{equation}

The key will be the following criterion from \cite{D}:

\begin{thm}\label{din}
$\alpha\geq 0$ iff $\bar zAz^T\geq 0$ for all 
$z\in\mathbb C^6$ with $z_1z_6+z_2z_5+z_3z_4=0$.
\end{thm}

\begin{proof}[Sketch of proof]
For $\gamma_1=b_1\omega_1+\dots+b_4\omega_4$, 
$\gamma_2=c_1\omega_1+\dots+c_4\omega_4$ we have
\begin{align*}
i\gamma_1\wedge\bar\gamma_1\wedge i\gamma_2\wedge\bar\gamma_2
  &=\sum_{j,k,l,m=1}^4b_j\bar b_kc_l\bar c_m\,\omega_j\wedge\omega_l
     \wedge\bar\omega_k\wedge\bar\omega_m\\
  &=\sum_{\substack{j<l\\ k<m}}(b_jc_l-b_lc_j)\overline{(b_kc_m-b_mc_k)}
     \,\omega_j\wedge\omega_l\wedge\bar\omega_k\wedge\bar\omega_m.
\end{align*}
It is now enough to show that the image of the mapping
\begin{align*}
\mathbb C^8\ni(b_1,\dots,b_4,c_1,\dots,c_4)\longmapsto&\\
  (b_1c_2-b_2c_1,b_1c_3-b_3c_1,b_1&c_4-b_4c_1,b_2c_3-b_3c_2,
      -b_2c_4+b_4c_2,b_3c_4-b_4c_3)\in\mathbb C^6
\end{align*}
is precisely $\{z\in\mathbb C^6:z_1z_6+z_2z_5+z_3z_4=0\}$.
\end{proof}

\pagebreak

Using Theorem \ref{din} and an idea from \cite{D} we can show:

\begin{prop}\label{prop}
The form
  $$\alpha_a=\sum_{j=1}^6\Omega_j\wedge\bar\Omega_j
     +a\,\Omega_1\wedge\bar\Omega_6+\bar a\,\Omega_6\wedge\bar\Omega_1$$
is positive if and only if $|a|\leq 2$.
\end{prop}

\begin{proof}
We have
  $$\bar zAz^T=|z|^2+2\Re(a\bar z_1z_6)
     \geq 2|z_1z_6|+2|z_2z_5+z_3z_4|+2\Re(a\bar z_1z_6).$$
If $z_1z_6+z_2z_5+z_3z_4=0$ and $|a|\leq 2$ we clearly get
$\bar zAz^T\geq 0$. If we take $z_1,z_6$ with 
$\bar z_1z_6=-\bar a$, $|z_1|=|z_6|$ and $z_2,\dots,z_5$ with
$z_2z_5+z_3z_4=-z_1z_6$ then $\bar zAz^T=2|a|(2-|a|)$.
\end{proof}

By \eqref{prod}
  $$\alpha_a\wedge\alpha_b=2(3+\Re(a\bar b))dV.$$
Therefore $\alpha_2,\alpha_{-2}$ are real and positive 
but $\alpha_2\wedge\alpha_{-2}<0$.

In view of Theorem \ref{din} we see that Theorem \ref{thm1}
is equivalent to the following:

\begin{thm}\label{thm2}
Let $A=(a_{jk})\in\mathbb C^{6\times 6}$ be hermitian and 
such that $\bar zAz^T\geq 0$ for $z\in\mathbb C^{6}$ with
$z_1z_6+z_2z_5+z_3z_4=0$. Then
\begin{equation*}
\sum_{j,k=1}^{6}a_{jk}\,a_{7-j,7-k}\geq 0.
\end{equation*}
\end{thm}

We will need the following technical reduction:

\begin{lem}\label{lem} 
For every $(2,2)$-form $\alpha$ in $\mathbb C^4$ we can find
a basis $\omega_1,\dots,\omega_4$ of $(\mathbb C^4)^\ast$ such that
\begin{equation}\label{cc1}
\alpha\wedge\omega_1\wedge\omega_2\wedge
     \bar\omega_1\wedge\bar\omega_j
    =\alpha\wedge\omega_1\wedge\omega_2\wedge
         \bar\omega_2\wedge\bar\omega_j=0
\end{equation}
for $j=3,4$.
\end{lem}

\begin{proof}
We may assume that $\alpha\neq 0$, then we can find 
$\omega_1,\omega_2\in(\mathbb C^4)^\ast$ such that
\begin{equation}\label{cc}
\alpha\wedge\omega_1\wedge\omega_2\wedge\bar\omega_1\wedge
\bar\omega_2=\alpha\wedge i\omega_1\wedge\bar\omega_1\wedge
i\omega_2\wedge\bar\omega_2\neq 0.
\end{equation}
By $V_1$ denote the subspace spanned by $\omega_1$, $\omega_2$ 
and by $V_2$ the subspace of all $\omega\in(\mathbb C^4)^\ast$ 
satisfying \eqref{cc1} with $\omega_j$ replaced by $\omega$.
Then $\dim V_1=2$, $\dim V_2\geq 2$ and by \eqref{cc} 
we infer $V_1\cap V_2=\{0\}$, hence 
$(\mathbb C^4)^\ast=V_1\oplus V_2$.
\end{proof}

\section{Proof of Theorem \ref{thm2}}\label{s2}

By Lemma \ref{lem} we may assume that the matrix from Theorem 
\ref{thm2} satisfies
  $$a_{26}=a_{36}=a_{46}=a_{56}=0$$
and
  $$a_{62}=a_{63}=a_{64}=a_{65}=0.$$
Then
  $$\sum_{j,k=1}^{6}a_{jk}\,a_{7-j,7-k}=
     \sum_{j,k=2}^{5}a_{jk}\,a_{7-j,7-k}
     +2(a_{11}a_{66}+|a_{16}|^2).$$
Therefore Theorem \ref{thm2} is in fact equivalent to the 
following result:

\begin{thm}\label{thm4}
Let $A=(a_{jk})\in\mathbb C^{4\times 4}$ be hermitian and 
such that $\bar zAz^T\geq 0$ for $z\in\mathbb C^{4}$ with
$z_1z_4+z_2z_3=0$. Then
\begin{equation}\label{aa}
\sum_{j,k=1}^{4}a_{jk}\,a_{5-j,5-k}\geq 0.
\end{equation}
\end{thm}

\begin{proof}
Write
  $$A=\left(\begin{matrix}
      \lambda_1&a&b&\alpha\\ 
      \bar a&\lambda_2&\beta&-c\\ 
      \bar b&\bar\beta&\lambda_3&-d\\ 
      \bar\alpha&-\bar c&-\bar d&\lambda_4
      \end{matrix}\right).$$
It satisfies the assumption of the theorem if and only if for
every $z\in\mathbb C^4$ of the form $z=(1,\zeta,w,-\zeta w)$
one has $\bar zAz^T\geq 0$. We can then compute
\begin{align*}
\bar zAz^T=&\lambda_1+2\Re(a\zeta)+\lambda_2|\zeta|^2\\
   &+2\Re[\big(b-\alpha\zeta+\beta\bar\zeta+c|\zeta|^2\big)w]\\
   &+\big(\lambda_3+2\Re(d\zeta)+\lambda_4|\zeta|^2\big)|w|^2.
\end{align*}
Therefore $A$ satisfies the assumption iff $\lambda_j\geq 0$,
\begin{equation}\label{war1}
|a|\leq\sqrt{\lambda_1\lambda_2},\ \ \ 
|b|\leq\sqrt{\lambda_1\lambda_3},\ \ \
|c|\leq\sqrt{\lambda_2\lambda_4},\ \ \
|d|\leq\sqrt{\lambda_3\lambda_4},
\end{equation}
and for every $\zeta\in\mathbb C$ 
\begin{equation}\label{war2}
\left|b-\alpha\zeta+\beta\bar\zeta+c|\zeta|^2\right|^2
\leq\left(\lambda_1+2\Re(a\zeta)+\lambda_2|\zeta|^2\right)
\left(\lambda_3+2\Re(d\zeta)+\lambda_4|\zeta|^2\right).
\end{equation}
In our case \eqref{aa} is equivalent to
  $$4\Re(a\bar d+b\bar c)\leq 
     2(\lambda_1\lambda_4+\lambda_2\lambda_3)
     +2(|\alpha|^2+|\beta|^2).$$
We will in fact prove something more:
\begin{equation}\label{aa2}
4\Re(a\bar d+b\bar c)\leq(\sqrt{\lambda_1\lambda_4}
   +\sqrt{\lambda_2\lambda_3})^2+(|\alpha|+|\beta|)^2.
\end{equation}

Without loss of generality we may assume that
  $$\Re(a\bar d)>0,\ \ \ \Re(b\bar c)>0,$$
for if for example $\Re(a\bar d)\leq 0$ then by \eqref{war1}
  $$4\Re(a\bar d+b\bar c)\leq 4\Re(b\bar c)
     \leq 4\sqrt{\lambda_1\lambda_2\lambda_3\lambda_4}
     \leq(\sqrt{\lambda_1\lambda_4}
        +\sqrt{\lambda_2\lambda_3})^2.$$
Set $u:=\Re(a\bar d)$ and $\zeta:=-r\bar d/|d|$, where $r>0$
will be determined later. Then the right-hand side of 
\eqref{war2} we can write as follows
\begin{align*}
\big(\lambda_1-&\frac{2ur}{|d|}+\lambda_2r^2\big)
  \big(\lambda_3-2r|d|+\lambda_4r^2\big)\\
 &=\left(\lambda_1+\lambda_2r^2\right)
    \left(\lambda_3+\lambda_4r^2\right)+4ur^2
    -2r\left[\lambda_1|d|+\lambda_3\frac{u}{|d|}
      +r^2\left(\lambda_2|d|+\lambda_4\frac{u}{|d|}\right)\right]\\
 &\leq\left(\lambda_1+\lambda_2r^2\right)
     \left(\lambda_3+\lambda_4r^2\right)+4ur^2
     -4r^2(\sqrt{\lambda_1\lambda_4}+\sqrt{\lambda_2\lambda_3})
            \sqrt u\\
  &=\big(\sqrt{\lambda_1\lambda_4}+\sqrt{\lambda_2\lambda_3}
            -2\sqrt u\big)^2r^2+\big(\sqrt{\lambda_1\lambda_3}
               -\sqrt{\lambda_2\lambda_4}\,r^2\big)^2.
\end{align*}
For $r=(\frac{\lambda_1\lambda_3}{\lambda_2\lambda_4})^{1/4}$
from \eqref{war2} we thus obtain
  $$\left|\frac br+\frac{\bar d}{|d|}\alpha-\frac{d}{|d|}\beta+cr\right|
     \leq\sqrt{\lambda_1\lambda_4}+\sqrt{\lambda_2\lambda_3}
          -2\sqrt u.$$
We also have
  $$\left|\frac br+\frac{\bar d}{|d|}(\alpha-\bar\beta)+cr\right|
    \geq \left|\frac br+cr\right|-|\alpha|-|\beta|\geq
    2\sqrt{\Re(b\bar c)}-|\alpha|-|\beta|$$
and therefore
  $$2\sqrt{\Re(a\bar d)}+2\sqrt{\Re(b\bar c)}
     \leq\sqrt{\lambda_1\lambda_4}+\sqrt{\lambda_2\lambda_3}
       +|\alpha|+|\beta|.$$
To get \eqref{aa2} we can now use the following fact: if $0\leq a_1\leq x$,
$0\leq a_2\leq x$ and $a_1+a_2\leq x+y$ then $a_1^2+a_2^2\leq
x^2+y^2$. This can be easily verified: if $a_1+a_2\leq x$ then
$a_1^2+a_2^2\leq x^2$ and if $a_1+a_2\geq x$ then
  $$x^2+y^2\geq x^2+(a_1+a_2-x)^2=a_1^2+a_2^2+2x(x-a_1)(x-a_2).$$
\end{proof}

\section{$(p,p)$-forms in $\mathbb C^{2p}$}\label{s3}

In $\mathbb C^{2p}$ we choose the positive volume form
  $$dV:=idz_1\wedge d\bar z_1\wedge\dots\wedge idz_{2p}
     \wedge d\bar z_{2p}=dz_1\wedge\dots\wedge dz_{2p}
     \wedge d\bar z_1\wedge\dots\wedge d\bar z_{2p}.$$
By $\mathcal I$ we will denote the set of subscripts $J=(j_1,\dots,
j_p)$ such that $1\leq j_1<\dots<j_p\leq 2p$. For every $J\in\mathcal I$ 
there exists unique $J'\in\mathcal I$ such that $J\cup J'=\{1,\dots,2p\}$.
We also denote $dz_J=dz_{j_1}\wedge\dots\wedge dz_{j_p}$
and $\varepsilon_J=\pm 1$ is defined in such a way that 
  $$dz_J\wedge dz_{J'}=\varepsilon_J\,dz_1\wedge\dots\wedge
     dz_{2p}.$$
Note that
\begin{equation}\label{eps}
\varepsilon_J\varepsilon_{J'}=(-1)^p.
\end{equation}

Every $(p,p)$-form $\alpha$ in $\mathbb C^{2p}$ we can associate with a
${N\times N}$-matrix $(a_{JK})$, where 
  $$N=\sharp\mathcal I=\frac{(2p)!}{(p!)^2},$$ 
by
\begin{equation}\label{mat}
\alpha=\sum_{J,K}a_{JK}idz_{j_1}\wedge d\bar z_{k_1}
     \wedge\dots\wedge idz_{j_p}\wedge d\bar z_{k_p}
     =i^{p^2}\sum_{J,K}a_{JK}dz_J\wedge d\bar z_K
\end{equation}
(note that $(-1)^{p(p-1)/2}i^p=i^{p^2}$). Then
\begin{equation}\label{sq}
\alpha^2=\sum_{J,K}\varepsilon_J\varepsilon_K a_{JK}a_{J'K'}\,dV
\end{equation}
and for $\gamma_1,\dots,\gamma_p\in(\mathbb C^{2p})^\ast$
  $$\alpha\wedge i\gamma_1\wedge\bar\gamma_1\wedge\dots\wedge 
     i\gamma_p\wedge\bar\gamma_p
     =\sum_{J,K}a_{JK}\,\gamma_1\wedge\dots\wedge\gamma_p\wedge dz_J
        \wedge\overline{(\gamma_1\wedge\dots\wedge\gamma_p\wedge dz_K)}.$$
Therefore $(a_{JK})$ has to be positive semi-definite on the image of 
Pl\"ucker embedding
\begin{equation}\label{pl}
((\mathbb C^{2p})^\ast)^p\ni(\gamma_1,\dots,\gamma_p)\longmapsto\left(\frac
     {\gamma_1\wedge\dots\wedge\gamma_p\wedge dz_J}{dz_1\wedge\dots\wedge
          dz_{2p}}\right)_{J\in
     \mathcal I}\in\mathbb C^{N}
\end{equation}
which is well known to be a variety in $\mathbb C^{N}$ (see e.g. \cite{H},
p. 64). 

We are now ready to prove Theorem \ref{thmp}:

\begin{proof}[Proof of Theorem \ref{thmp}]
First note that it is enough to show it for $p=3$. For if $\alpha$ is a 
$(3,3)$-form in $\mathbb C^6$ such that $\alpha\geq 0$ and $\alpha^2<0$
then for $p>3$ we set 
  $$\beta:=idz_7\wedge d\bar z_7+\dots+idz_{2p}\wedge d\bar z_{2p}.$$
We now have $\alpha\wedge\beta^{p-3}\geq 0$ but $(\alpha\wedge\beta^{p-3})^2=
\alpha^2\wedge\beta^{2p-6}<0$.

Set $p=3$, so that $N=20$, and order $\mathcal I=\{J_1,\dots,J_{20}\}$ 
lexicographically. Then the image of Pl\"ucker embedding is in particular 
contained in the quadric
\begin{equation}\label{quad}
z_1z_{20}-z_{10}z_{11}+z_5z_{16}-z_2z_{19}=0.
\end{equation}
For positive $a,\lambda,\mu$ to be determined later define
  $$\alpha:=\lambda i\big(dz_{J_1}\wedge d\bar z_{J_1}+dz_{J_{20}}
      \wedge d\bar z_{J_{20}}\big)
     +\mu i\!\!\!\!\!\sum_{\substack{k\in\{2,5,10,\\ \ \ 11,16,19\}}}
     dz_{J_k}\wedge d\bar z_{J_k}
     +a\big(dz_{J_1}\wedge d\bar z_{J_{20}}+
          dz_{J_{20}}\wedge d\bar z_{J_1}\big).$$
Then, similarly as in the proof of Proposition \ref{prop},
\begin{align*}
\bar zAz^T&=\lambda(|z_1|^2+|z_{20}|^2)
     +\mu(|z_2|^2+|z_5|^2+|z_{10}|^2+|z_{11}|^2+|z_{16}|^2+|z_{19}|^2)
      +2a\Re(\bar z_1z_{20})\\
  &\geq 2(\lambda-a)|z_1z_{20}|+2\mu|-z_{10}z_{11}+z_5z_{16}-z_2z_{19}|\\
  &=2(\lambda+\mu-a)|z_1z_{20}|
\end{align*}
if $z$ satisfies \eqref{quad}. Therefore $\alpha\geq 0$ if 
$a\leq\lambda+\mu$. 

On the other hand by \eqref{sq} and \eqref{eps}
  $$\alpha^2=2(\lambda^2+3\mu^2-a^2)dV.$$
We see that if we take $a=\lambda+\mu$ and $\lambda>\mu>0$ then
$\alpha\geq 0$ but $\alpha^2<0$.
\end{proof}


\begin{thebibliography}{3}

\bibitem{BT} {\sc E. Bedford, B.A. Taylor}, {\sl Simple and positive
vectors in the exterior algebra of $\mathbb C^n$}, preprint, 
1974

\bibitem{De} {\sc J.-P. Demailly},  {\sl Complex Analytic and
Differential Geometry}, monograph, 1997, available at
{\tt http://www-fourier.ujf-grenoble.fr/$\widetilde{\phantom{a}}$demailly}

\bibitem{D} {\sc S. Dinew}, {\sl On positive $\mathbb C_{(2,2)}
(\mathbb C^4)$ forms}, preprint, 2006

\bibitem{H} {\sc J. Harris}, {\sl Algebraic Geometry. A First Course}, 
Graduate Texts in Math. 133, Springer, 1995

\bibitem{HK} {\sc R. Harvey, A.W. Knapp}, {\sl Positive $(p,p)$ forms,
Wirtinger's inequality, and currents}, Value Distribution Theory,
R.O. Kujala and A.L. Vitter III (ed.), Part A, pp. 43--62,
Dekker, 1974

\end{thebibliography}
\end{document}